\documentclass[12pt,draft,onecolumn]{IEEEtran}

\usepackage{amsmath}
\usepackage{amsthm}
\usepackage{amsfonts}
\usepackage{amssymb}
\usepackage{mathrsfs}
\usepackage{url}
\usepackage[usenames,dvipsnames]{xcolor}
\usepackage{enumerate}

\newcommand{\R}{\mathbb{R}}
\newcommand{\C}{\mathbb{C}}
\newcommand{\F}{\mathbb{F}}
\newcommand{\rk}{\mathrm{rk}}
\newcommand{\diag}{\mathrm{diag}}

\theoremstyle{definition}
\newtheorem{Lemma}{Lemma}

\newtheorem{Theorem}{Theorem}
\newtheorem{Remark}{Remark}
\newtheorem{Definition}{Definition}

\definecolor{Royalblue}{cmyk}{1,0.30,0.2,0.2}

\author{Giacomo Baggio, Augusto Ferrante%
\thanks{Giacomo Baggio is with the Dipartimento di Ingegneria dell'€™Informazione, Universit\`a di Padova, via Gradenigo, 6/B€" I-35131 Padova, Italy. E-mail: giacomo.baggio@studenti.unipd.it.
Augusto Ferrante is with the Dipartimento di Ingegneria dell€'Informazione, Universit\`a di Padova, via Gradenigo, 6/B€" I-35131 Padova, Italy. E-mail: augusto@dei.unipd.it.}}

\begin{document}

\title{On Minimal Spectral Factors  with Zeroes and Poles lying on Prescribed Regions}

\maketitle

\begin{abstract} 
In this paper, we consider a general discrete-time spectral factorization problem for rational matrix-valued functions.
We build on a recent result establishing existence of a spectral factor  whose zeroes and poles lie in any pair of prescribed regions of the complex plane featuring a geometry compatible with symplectic symmetry.
In this general setting, uniqueness of the spectral factor  is not guaranteed. It was, however, conjectured that if we further impose stochastic minimality, uniqueness can be recovered.
The main result of his paper is a proof of this conjecture.
\end{abstract}

\section{Introduction and problem definition}

Spectral factorization is a crucial problem of many areas of systems and control theory, from LQ optimal control theory \cite{Willems-1971} to filtering and estimation theory \cite{Lindquist-P-85-siam,Lindquist-P-91-jmsec}, to cite but a few. A seminal paper in spectral factorization theory is due to  Youla  \cite{Youla-1961}. In that paper, a  constructive procedure was established to  compute a stochastically minimal outer spectral factor of a given rational spectral density $\Phi(s)$ without requiring any additional system-theoretic assumption.  
In fact, Youla's only assumptions | that are clearly necessary for the existence of a spectral factor | are that the real rational spectrum $\Phi(s)$ is positive semi-definite in the points of the imaginary axis where it is finite and  features the {\em Hamiltonian paraconjugate symmetry}, i.e. $\Phi(s)= [\Phi(-s)]^\top$. Under these assumptions, Youla has established existence, and provided a procedure for construction, of a rational matrix function $W(s)$ | a {\em spectral factor} | analytic with its right-inverse in the open right half complex plane, such that  $\Phi(s)=[W(-s)]^\top W(s)$.

In \cite{BF-2015} a discrete-time counterpart of the Youla's result is established as a corollary of a much more general result that allows for the selection of the analyticity regions of the spectral factor  and  of its right-inverse. Remarkably, the latter feature is of key interest in stochastic realization and a-causal estimation theory, see, e.g., \cite{Colaneri-F-siam, Picci-P-94}.
The main result of \cite{BF-2015}, that may be viewed as the starting point for this note, may be described as follows.
Let $\Phi(z)$ be a real, rational, matrix-valued function.
Assume that $\Phi(z)$ is positive semi-definite in the points of the unit circle where it is finite and that  it features the {\em Symplectic paraconjugate symmetry}, i.e.
$\Phi(z)= [\Phi(1/z)]^\top$. 
Let $\mathscr{A}_p$ and $\mathscr{A}_z$ be regions of the extended complex plane $\overline{\C}$, compatible with symplectic structure (i.e. for each $z\in\overline{\C}$ with $|z|\neq 1$, exactly one element of the pair $(z,z^{-1})$ is in $\mathscr{A}_p$  and $z\not\in\mathscr{A}_p$ if $|z|= 1$; and the same holds for $\mathscr{A}_z$).
Then, there exists a rational matrix function $W(z)$
| a {\em spectral factor} | analytic in  $\mathscr{A}_p$ and with  right-inverse analytic in $\mathscr{A}_z$, such that $\Phi(z)=[W(1/z)]^\top W(z)$.
In the case when $\mathscr{A}_p$ and $\mathscr{A}_z$ coincide with the subset
of $\overline{\C}$ of the $z$ that are outside the closed unit disc, we get the discrete-time counterpart of the result of Youla.
There are, however,  many other interesting situations. For example in backward filtering, $\mathscr{A}_p$ is fixed by the system's dynamics  while $\mathscr{A}_z$ is the open unit disc.

From this general result, a very interesting question arises. In fact, when $\mathscr{A}_p=\mathscr{A}_z=\mathscr{A}$, it is not difficult to see that the corresponding spectral factor is essentially unique (i.e. unique up to multiplication on the left side by a constant orthogonal matrix): the key idea is that, starting from a reference spectral factor $W(z)$, a second spectral factor $W_1(z)$ must be of the form  $W_1(z)=Q(z)W(z)$ with $Q(z)$ being all-pass so that if $Q(z)$ has a pole in $p$, it necessarily has  a zero in $1/p$; therefore, for any non-constant $Q(z)$, either $W_1(z)$ or its right-inverse is no longer analytic in $\mathscr{A}$.
On the contrary, when $\mathscr{A}_p\neq\mathscr{A}_z$, we can easily obtain a spectral factor $W_1(z)$ with the prescribed analyticity properties by selecting an all-pass function $Q(z)$ featuring poles in $\overline{\C}\setminus \mathscr{A}_p$ and zeroes in  $\overline{\C}\setminus \mathscr{A}_z$.
Thus, there appears to be an inherent ambiguity in the choice of the spectral factor in this general case.
In this paper, we show that this is in not the case if we further impose that the spectral factor has minimal complexity as measured by its McMillan degree. In fact, we will show that, under this assumption, for any choice of the analyticity regions  $\mathscr{A}_p$ and $\mathscr{A}_z$, the spectral factor is essentially unique.
In the scalar case, this result is straightforward. In the general matricial case, however, a rational function can feature a pole and a zero in the same point so that the result appears to be quite difficult to derive and was left open as a conjecture in \cite{BF-2015}.
Our proof makes use of a very elegant and profound parametrization of rational all-pass functions established by Alpay and Gohberg in \cite{unitary2}.

{\noindent \bfseries Paper structure.} The paper is organized as follows: In \S II, we review some preliminary notions of rational matrix theory and we introduce some ancillary results. In \S III, we present our main theorem. Finally, in \S IV, we draw some concluding remarks and we list a number of possible future research directions.  

{\noindent \bfseries Notation.} In what follows, we write $G^\top$, $G^*$, $G^{-1}$, $G^{-R}$ for the transpose, Hermitian conjugate, inverse and right inverse of matrix $G$, respectively. As usual, $I_n$ is the $n\times n$ identity matrix and $\mathrm{diag}[a_1,\dots,a_n]$ stands for the matrix whose diagonal entries are $a_1,\dots,a_n$. We let $\mathbb{T}:=\{\,z\in\mathbb{C} \, :\, |z|=1\,\}$  and we denote by $\overline{\C}:=\C\cup \{\infty\}$ the extended complex plane. 
We denote by $\F[z]^{m\times n}$ and $\F(z)^{m\times n}$ the set of $m\times n$ polynomial and rational matrices with coefficients in the field $\F$ (we consider the two cases $\F=\R,\C$). A polynomial matrix $G(z)\in\F[z]^{m\times n}$ is said to be unimodular if it possesses a polynomial inverse (either left, right or both). Notably, a square polynomial matrix $G(z)$ is unimodular if and only if $\det G(z)$ is a constant. Given a rational matrix $G(z)\in\C(z)^{m\times n}$, we let $G^*(z):=[G(1/\overline{z})]^*$, where $\overline{z}$ is the complex conjugate of $z\in\C$, and we denote by $\rk(G)$ the normal rank of $G(z)$, i.e., the rank almost everywhere in $z\in\C$ of $G(z)$. The rational matrix $G(z)$ is said to be analytic in a region of the complex plane if all its entries are analytic in this region.  If  $\rk(G)=m$, then $G^{-R}(z)$ denotes a ``minimal''  right inverse of $G(z)$, i.e., a right inverse of $G(z)$ whose poles coincide with the zeroes of $G(z)$. Other standard notation and terminology is taken from \cite{BF-2015}.

\section{Preliminary results}

If $G(z)\in\C(z)^{m\times n}$, $\rk(G)=r\leq \min\{m,n\}$, by the Smith-McMillan Theorem \cite[Ch.6, \S 5]{Kailath}, there exist unimodular matrices $U(z)\in\C[z]^{m\times r}$ and $V(z)\in\C[z]^{r\times n}$ such that
\begin{align}\label{eq:smith-mcmillan-canonic-form}
	D(z):	& =		U(z)G(z)V(z)\nonumber \\
		 	& =		\diag\left[\frac{\varepsilon_1(z)}{\psi_1(z)},\frac{\varepsilon_2(z)}{\psi_2(z)},\dots,\frac{\varepsilon_r(z)}{\psi_r(z)}\right],
\end{align}
where $\varepsilon_1(z),\,\varepsilon_2(z),\,\dots,\,\varepsilon_r(z),\, \psi_1(z),\, \psi_2(z),\,\dots,\,\psi_r(z)\in\C[z]$ are monic polynomials satisfying the conditions: {(i)} $\varepsilon_i(z)$ and $\psi_i(z)$ are relatively prime, $i=1,2,\dots,r$, {(ii)} $\varepsilon_i(z)\mid \varepsilon_{i+1}(z)$ and $\psi_{i+1}(z)\mid \psi_{i}(z)$, $i=1,2,\dots,r-1$.\footnote{ If $p(z)$ and $q(z)$ are two polynomials in $\R[z]$, $p(z) \mid q(z)$ means that $p(z)$ divides $q(z)$.}

The rational matrix $D(z)$ in (\ref{eq:smith-mcmillan-canonic-form}) is known as the Smith-McMillan (SM, for short) canonical form of $G(z)$. The (finite) zeroes of $G(z)$ coincide with the zeroes of $\varepsilon_r(z)$ and the (finite) poles of $G(z)$ with the zeroes of $\psi_1(z)$. The degree of a pole and zero at $\alpha\in\C$ (denoted by $\delta_p(G;\alpha)$ and $\delta_z(G;\alpha)$, respectively) is equal to the sum of the degrees of the zero at $\alpha$ of all the $\psi_i(z)$ and of all the $\varepsilon_i(z)$, respectively.\footnote{If $\alpha=\infty$, then we can consider the transformation $z\mapsto\lambda^{-1}$ and the definition still applies by considering the degree of the pole/zero at $\lambda=0$ of $G(\lambda)$.} If $G(z)$ has no pole (zero) at $\alpha$, we let $\delta_p(G;\alpha)=0$ ($\delta_z(G;\alpha)=0$). Furthermore, if $p_1,\dots,p_h$ are the distinct poles (the pole at infinity included) of $G(z)$, the McMillan degree of $G(z)$ is defined as \cite[Ch.6, \S5]{Kailath}
\begin{equation}\label{eq:mcmillan-degree}
	\delta_M(G) := \sum_{i=1}^h\delta_p(G;p_i).
\end{equation} 

\begin{Definition} Let $G(z)\in\C(z)^{m\times n}$, $H(z)\in\C(z)^{n\times p}$ and $\alpha\in\overline{\C}$. We say that in the product $G(z)H(z)$ there is: 
\begin{enumerate}
\item a pole cancellation at $\alpha$ if $\delta_p(GH;\alpha)<\delta_p(G;\alpha)+\delta_p(H;\alpha)$; 
\item a zero cancellation at $\alpha$ if $\delta_z(GH;\alpha)<\delta_z(G;\alpha)+\delta_z(H;\alpha)$; 
\item a zero-pole cancellation at $\alpha$ if both conditions 1) and 2) are met. 
\end{enumerate}
\end{Definition}

\begin{Remark}
If $\mathrm{rk}(G)=\mathrm{rk}(H)=n$ then a zero or pole cancellation at $\alpha$ in the product $G(z)H(z)$ always corresponds to a zero-pole cancellation at $\alpha$. A proof of this fact is postponed to the end of this section (Lemma \ref{Lemma-zero-pole-cancellation}).

 However, in general, 1) and 2) are not equivalent. Indeed, consider for instance the product

\[
	G(z)H(z)=\begin{bmatrix}
						1 & -1
			\end{bmatrix}
			\begin{bmatrix}
						\frac{2z+3}{(z+1)(z+2)} \\
						\frac{1}{z+2}
			\end{bmatrix}
			=\frac{1}{z+1}
\]  
and observe that there is a pole cancellation at $-2$ which does not correspond to a zero-pole cancellation at $-2$. 
\end{Remark}

Two special classes of rational matrices are the following ones.

\begin{Definition}[Paraconjugate-Hermitian]
A rational matrix $G(z)\in\C(z)^{n\times n}$ is said to be paraconjugate-Hermitian if $G(z) = G^*(z)$.
\end{Definition}

\begin{Definition}[Paraconjugate-unitary or All-Pass]
A rational matrix $G(z)\in\C(z)^{r\times r}$ is said to be paraconjugate-unitary or  all-pass if 
\[
	G^*(z) G(z) = G(z) G^*(z)=I_n.
\]
\end{Definition}

\begin{Remark}
A real paraconjugate-Hermitian (paraconjugate-unitary) matrix $G(z)\in\R(z)^{r\times r}$ is said to be para-Hermitian (para-unitary, respectively). In addition, it is worth noting that a paraconjugate-Hermitian matrix is Hermitian in the ordinary sense upon the unit circle, while a  paraconjugate-unitary matrix is unitary in the ordinary sense upon the unit circle. 
\end{Remark}

A useful characterization of the class of paraconjugate-unitary matrices is provided by the following Lemma.

\begin{Lemma}\label{Lemma1}
Let $V(z)\in\mathbb{C}(z)^{r\times r}$, $\delta_M(V)=n$, and let $\{\alpha_i\}_{i=1}^n$ be the poles of $V(z)$ counted with multiplicity, then $V(z)$ is paraconjugate-unitary if and only if it can be written as
\begin{equation}\label{eq:paraconjugateunitary}
	V(z) = U U_1(z)U_2(z)\cdots U_n(z),
\end{equation}
with $U\in\C^{r\times r}$ being constant unitary and\footnote{We adopt the convention $\frac{1-\overline{\alpha}_iz}{z-\alpha_i}=:z$ if $\alpha_i=\infty$.}
\begin{align}\label{eq:Uifactorlemma}
	U_i(z):=I_r+\left(\frac{1-\overline{\alpha}_iz}{z-\alpha_i}-1\right)P_i, \quad \alpha_i\in\overline{\mathbb{C}}\setminus \mathbb{T},
\end{align}
with $P_i\in\mathbb{C}^{r\times r}$ being an orthogonal rank-one projection. Moreover, the product in the right-hand side of (\ref{eq:paraconjugateunitary}) is minimal, i.e., $\delta_M(V)=\delta_M(U_1)+\cdots+\delta_M(U_n)$.
\end{Lemma}
\begin{proof}
See \cite[Thm. 3.12]{unitary2}.
\end{proof}

\begin{Remark} Given any decomposition of a paraconjugate-unitary matrix $V(z)$ of the form in (\ref{eq:Uifactorlemma}), we have that:
\begin{enumerate}
\item Every factor $U_i(z)$ in (\ref{eq:Uifactorlemma}) is paraconjugate-unitary. Indeed, by direct computation:
\begin{align*}
	 U_i^*(z)U_i(z) & = I_r + \left(\frac{1-\alpha_iz^{-1}}{z^{-1}-\overline{\alpha}_i}-1\right)P_i + \left(\frac{1-\overline{\alpha}_iz}{z-\alpha_i}-1\right)P_i \, +\\
	 &\hspace{6.2cm}+ \left(\frac{1-\alpha_iz^{-1}}{z^{-1}-\overline{\alpha}_i}-1\right)\left(\frac{1-\overline{\alpha}_iz}{z-\alpha_i}-1\right)P_i\\
	& =	I_r-2P_i+\frac{1-\alpha_iz^{-1}}{z^{-1}-\overline{\alpha}_i}P_i+\frac{1-\overline{\alpha}_iz}{z-\alpha_i}P_i+2P_i-\frac{1-\alpha_iz^{-1}}{z^{-1}-\overline{\alpha}_i}P_i-\frac{1-\overline{\alpha}_iz}{z-\alpha_i}P_i\\
	& =I_r.
\end{align*}
\item Every pole at $\alpha_i$ of $V(z)$ of degree $d_i$ is accompanied by a zero of $V(z)$ at $1/\overline{\alpha}_i$ of the same degree. In particular, if $\alpha_i\neq \infty$, the SM canonical form of $U_i(z)$ in (\ref{eq:Uifactorlemma}) is given by
\begin{align*}
	\mathrm{diag}\left[\frac{1}{z-\alpha_i},1, \dots, 1, z-1/\overline{\alpha}_i\right].
\end{align*} 
\item Since the decomposition is minimal and $\delta_M(U_i)=1$, $i=1,\dots,n$, it follows that 
\[
	\delta_M(V)=\sum_{i=1}^n \delta_M(U_i)=n.
\]
\item Since the orthogonal rank-one projection $P_i$ in (\ref{eq:Uifactorlemma}) can be written as $P_i=v_i v_i^*$ with $v_i\in\C^r$ s.t. $\|v_i\|^2=v_i^*v_i=1$, it holds\footnote{In the derivation we exploit the fact that if $A\in\mathbb{C}^{n\times r}$ and $B\in\mathbb{C}^{r\times n}$, then $\det(I_n+AB)=\det(I_r+BA)$.}
\begin{align*}
	\det U_i(z) 	& =	\det\left[I_r+\left(\frac{1-\overline{\alpha}_iz}{z-\alpha_i}-1\right)P_i\right]\\
				& =	\det\left[I_r+\left(\frac{1-\overline{\alpha}_iz}{z-\alpha_i}-1\right)v_i v_i^*\right]\\
				& =	\det\left[1+v_i^*\left(\frac{1-\overline{\alpha}_iz}{z-\alpha_i}-1\right)v_i \right]\\
				& =	\frac{1-\overline{\alpha}_iz}{z-\alpha_i}.
\end{align*}

\end{enumerate}
\end{Remark}

For the sake of completeness, we state and prove below two additional instrumental Lemmata.

\begin{Lemma}\label{Lemma-zero-pole-cancellation}
Let $G(z)\in\C(z)^{n\times r}$ and $H(z)\in\C(z)^{r\times m}$ with $\rk(G)=\rk(H)=r$. If $G(z)H(z)$ has a zero or pole cancellation at $\alpha\in\C$, then $G(z)H(z)$ has a zero-pole cancellation at $\alpha$.
\end{Lemma}

\begin{proof}
Assume that $G(z)H(z)$ has a pole cancellation at $\alpha\in\C$ (the proof for the case of a zero cancellation at $\alpha\in\C$ goes along the same lines).

Let $D(z),\, D'(z)\in\C(z)^{r\times r}$ be the SM canonical form of $G(z)$, $H(z)$, respectively. We can write 
\[
	G(z)=C(z)D(z)F(z) \quad \text{and}\quad H(z)=C'(z)D'(z)F'(z)
\]
with $F(z),\,C'(z)\in\C[z]^{r\times r}$, $C(z)\in\C[z]^{n\times r}$ and $F'(z)\in\C[z]^{r\times m}$ unimodular matrices. Hence the product $G(z)H(z)$ can be written as
\[
	G(z)H(z)=C(z)D(z)F(z)C'(z)D'(z)F'(z)
\]
where $M(z):=F(z)C'(z)\in\C[z]^{r\times r}$ is unimodular. Notice that, by virtue of the unimodularity of $C(z)$ and $F'(z)$, the SM canonical form of $G(z)H(z)$, denoted by $\Delta(z)$, coincides with that of $D(z)M(z)D'(z)$ (see \cite[Ex.6.5-6]{Kailath}). Moreover observe that, since $\rk(G)=\rk(H)=r$, then $\rk(DMD')=r$. Therefore, by taking determinants, we have
\begin{align}\label{eq:deter}
	\det \Delta(z) 	&	=	c\det D(z)\det D'(z)\nonumber\\
						& 	=	c \frac{n(z)}{d(z)} \frac{(z-\alpha)^{\delta_z(G;\alpha)+\delta_z(H;\alpha)}}{(z-\alpha)^{\delta_p(G;\alpha)+\delta_p(H;\alpha)}}
\end{align}
with $n(z)$ and $d(z)$ relatively prime polynomials s.t. $n(\alpha)\neq 0$, $d(\alpha)\neq 0$, and $c\in\C, \ c\neq 0$. On the other hand, since $\Delta(z)$ is the SM canonical form of $G(z)H(z)$, we get
\begin{align}\label{eq:deter-2}
	\det \Delta(z) 	= c \frac{n(z)}{d(z)} \frac{(z-\alpha)^{\delta_z(GH;\alpha)}}{(z-\alpha)^{\delta_p(GH;\alpha)}}.
\end{align}
Hence, a comparison of (\ref{eq:deter}) and (\ref{eq:deter-2}) yields
\begin{align}\label{eq:zero-pole-canc}
	 & \delta_p(GH;\alpha)- \delta_p(G;\alpha) -\delta_p(H;\alpha) = \delta_z(GH;\alpha)- \delta_z(G;\alpha)-\delta_z(H;\alpha).
\end{align}
Since, by assumption, $G(z)H(z)$ has a pole cancellation at $\alpha$, the left-hand side of (\ref{eq:zero-pole-canc}) is strictly negative. This in turn implies that the right-hand side of (\ref{eq:zero-pole-canc}) is strictly negative, i.e. $G(z)H(z)$ has a zero cancellation at $\alpha$. From this fact the thesis follows.
\end{proof}

\begin{Lemma}\label{Lemma-sum-polar-degree}
Let $G(z)\in\C(z)^{n\times r}$ and $H(z)\in\C(z)^{r\times m}$ with $\rk(G)=\rk(H)=r$. If $G(z)$ and $H(z)$ have no zeroes at $\alpha\in\C$ then 
\[
	\delta_p(GH;\alpha)=\delta_p(G;\alpha)+\delta_p(H;\alpha).
\]
\end{Lemma}
\begin{proof}
By following verbatim the first part of the proof of Lemma \ref{Lemma-zero-pole-cancellation}, we arrive at the expression
\begin{align}\label{eq:deter1}
	\det\Delta(z) & = \det D(z) \det M(z) \det D'(z) \nonumber\\ 
					   & = c\det D(z)\det D'(z), \quad c\in \C,\ c\neq 0.
\end{align}
Since by assumption $G(z)$ and $H(z)$ have no zero at $\alpha$, then $D(z)$ and $D'(z)$ have no zero at $\alpha$. Furthermore, $\Delta(z)$ has no zero at $\alpha$. This fact can be seen by taking the inverse of $\Delta(z)$, namely 
\[
	\Delta^{-1}(z)=D'^{-1}(z)M^{-1}(z)D^{-1}(z),
\] 
and by noting that the latter has no pole at $\alpha$, since the entries of $D^{-1}(z)$, $M^{-1}(z)$ and $D'^{-1}(z)$ do not have any pole at $\alpha$. This in turn implies that $\delta_p(D;\alpha)$, $\delta_p(D';\alpha)$ and $\delta_p(\Delta;\alpha)$ coincide with the degree of the pole at $\alpha$ in $\det D(z)$, $\det D'(z)$ and $\det \Delta(z)$, respectively. Hence, by summing up all the previous considerations, we get
\begin{align*}
	\delta_p(GH;\alpha)	& = \delta_p(\Delta;\alpha) \\
					& = \delta_p(\det \Delta;\alpha) \\
					& \stackrel{(\ref{eq:deter1})}{=} \delta_p(\det D;\alpha)+\delta_p(\det D';\alpha) \\
					& = \delta_p(D;\alpha)+\delta_p(D';\alpha) \\
					& = \delta_p(G;\alpha)+\delta_p(H;\alpha) 
\end{align*}
which concludes the proof.
\end{proof}

\begin{Remark}
Notice that Lemmata \ref{Lemma-zero-pole-cancellation} and \ref{Lemma-sum-polar-degree} still hold when $\alpha=\infty$. As a matter of fact, in this case, we can apply the change of variable $z\mapsto \lambda^{-1}$ and then consider the (degree of the) zero/pole at $\lambda=0$ in $G(\lambda)$ and $H(\lambda)$. 
\end{Remark}

\section{The main theorem}

Before proceeding with the proof of the main Theorem, we introduce some preliminary definitions.

\begin{Definition}[Spectrum]
A para-Hermitian rational matrix $\Phi(z)\in\mathbb{R}(z)^{n\times n}$ is said to be a spectrum if $\Phi(e^{j\omega})$ is positive semi-definite for all $\omega\in[0,2\pi)$ such that $\Phi(e^{j\omega})$ is defined.
\end{Definition}

\begin{Definition}[(Stochastically minimal) Spectral factor]
Given a spectrum $\Phi(z)\in\mathbb{R}(z)^{n\times n}$, a real matrix-valued function $W(z)$ satisfying
\[
	\Phi(z)=W^*(z)W(z),
\] 
is called a spectral factor of $\Phi(z)$. Moreover, the spectral factor $W(z)$ is said to be stochastically minimal if 
\[
	\delta_{M}(W)=\frac{1}{2}\delta_{M}(\Phi).
\]
\end{Definition}

Stochastically minimal spectral factors correspond to solutions of minimal complexity (i.e. minimal McMillan degree) of the spectral factorization problem. Indeed, a spectral factor $W(z)$ of $\Phi(z)$ cannot have McMillan degree smaller than $\frac{1}{2}\delta_{M}(\Phi)$.

\begin{Definition}[(Weakly) Unmixed-symplectic]\label{def:unmixed-symplectic}
A set $\mathscr{A}\subset \overline{\C}$ is unmixed-symplectic if 
\[
	\mathscr{A}\cup \mathscr{A}^*=\overline{\C}\setminus \{\,z\in\C\,:\, |z|=1\,\},\ \ \text{and}\ \ 
	\mathscr{A}\cap \mathscr{A}^*=\emptyset,
\]
where $\mathscr{A}^*=\{\,z\,:\, z^{-1}\in\mathscr{A}\,\}$.
The set $\mathscr{A}\subset \overline{\C}$ is weakly unmixed-symplectic if 
\[
	\mathscr{A}\cup \mathscr{A}^*=\overline{\C},\ \ \text{and}\ \ 
	\mathscr{A}\cap \mathscr{A}^*=\{\,z\in\C\,:\, |z|=1\,\},
\]
\end{Definition}

The following Theorem is the main result of this note.

\begin{Theorem}\label{thm1}
Let $\Phi(z)\in\mathbb{R}(z)^{n \times n}$ be a spectrum with $\rk(\Phi)=r\leq n$, $r \neq 0$. Let $W(z),W_1(z)\in\mathbb{R}(z)^{r\times n}$ be such that
\begin{enumerate}
\item $W(z)$ and $W_{1}(z)$ are spectral factors of $\Phi(z)$, i.e. $\Phi(z)=W^*(z)W(z)=W_1^*(z)W_1(z)$;
\item $W(z), W_1(z)$ are analytic in $\mathscr{A}_p$ and $W^{-R}(z),W_1^{-R}(z)$ are analytic in $\mathscr{A}_z$, where $\mathscr{A}_p$, $\mathscr{A}_z$ are weakly unmixed-symplectic regions;
\item $W(z)$ and $W_{1}(z)$ are stochastically minimal, i.e. $\delta_M(W)=\delta_M(W_1)=\frac{1}{2}\delta_M(\Phi)$.
\end{enumerate}
Then, $W_1(z)=TW(z)$ with $T\in\mathbb{R}^{r\times r}$ constant orthogonal.
\end{Theorem}

\begin{proof}
Before illustrating the details of the proof, we outline the key steps in order to provide a road-map that may help the reader. 
\begin{enumerate}[i)]
\item We consider the para-unitary function $T(z)$ satisfying $W_1(z)=T(z)W(z)$ and we show that $T(z)$ must have no poles and zeroes in the region $\mathscr{A}_{p}\cap \mathscr{A}_{z}$. 
\item We then assume by contradiction that $T(z)$ is non-constant and, more precisely, that $T(z)$ possesses poles both in $\mathscr{A}_{p}\setminus\mathscr{A}_z$ and in $\mathscr{A}_{z}\setminus\mathscr{A}_p$. 
\item We decompose $T(z)$ according to Lemma \ref{Lemma1} and, by exploiting the properties of this decomposition, we show that for each pole $\alpha\in\mathscr{A}_{p}\setminus\mathscr{A}_z$ of $T(z)$ there is a zero-pole cancellation both at $\alpha$ and at $1/\overline{\alpha}$ in the product $T(z)W(z)$. Hence, we arrive at the contradiction that there exists a spectral factor of $\Phi(z)$, say $\tilde{W}(z)$, such that $\delta_{M}(\tilde{W})<\frac{1}{2}\delta(\Phi)$. Since this is not possible, we conclude that $T(z)$ must have no poles in the region $\mathscr{A}_{p}\setminus\mathscr{A}_z$.
\item Finally, we exploit the fact that, by point 3), $W(z)$ and $W_{1}(z)$ are stochastically minimal spectral factors to conclude that $T(z)$ must have no poles in the region $\mathscr{A}_{z}\setminus\mathscr{A}_p$. This implies that $T(z)$ is a constant and orthogonal matrix.
\end{enumerate}

We now describe the details. Consider the matrix
\[
T(z):=W_1(z)W^{-R}(z).
\]
By taking into account Property 1), it is immediate to see that
$T^*(z)T(z)=I$, i.e. that $T(z)$  is para-unitary. Moreover, since the inverse of $T(z)$ is given by
\[
T^*(z)=T^{-1}(z)=W(z)W_1^{-R}(z),
\]
it follows that $T(z)$ is analytic with its inverse in $\mathscr{A}_z\cap \mathscr{A}_p$.
Now observe that 
\begin{align}\label{eq:W_1TW}
W_1(z)=T(z)W(z).
\end{align}
To see this, set $Z(z):=W_1(z)-T(z)W(z)$.
By employing again Property 1), it is immediate to see that
$Z^*(z)Z(z)=0$ so that $Z(z)$ is identically zero in the unit circle and, eventually, $Z(z)=0$.
We need to show that $T(z)$ is constant.

Assume, ab absurdo, that $T(z)$ has McMillan degree $d$ with poles $\alpha_1,\dots,\alpha_n$ of degree $m_1,\dots,m_n$ ($d=m_1+\cdots+m_n$), respectively, s.t. $\alpha_1,\dots,\alpha_t\in\mathscr{A}_p\setminus\mathscr{A}_z$ and $\alpha_{t+1},\dots,\alpha_n\in\mathscr{A}_z\setminus\mathscr{A}_p$. In what follows we assume that $\alpha_{i}\neq \infty$ for $i=1,\dots,n$. As a matter of fact, if this is not the case, we can always find a suitable M\"obius transformation $z\mapsto f(z)$ such that $T(f(z))$ has only finite poles. Thus, by considering this transformation, the argument in the proof still applies.

By exploiting Lemma \ref{Lemma1}, we can decompose $T(z)$ as
\begin{align}\label{eq:T(z)}
T(z)  =  & \  U U_{\alpha_1,1}(z)\cdots U_{\alpha_1,m_1-1}(z)U_{\alpha_n}(z)\cdots U_{\alpha_2}(z) U_{\alpha_1,m_1}(z),
\end{align}
with $U\in\C^{r\times r}$ constant unitary and
\begin{align}
U_{\alpha_i,j}(z)&:=I_r+\left(\frac{1-\overline{\alpha}_iz}{z-\alpha_i}-1\right)P_{i,j},\label{eq:Ualphaij}\\
U_{\alpha_i}(z)&:=U_{\alpha_i,1}(z)\cdots U_{\alpha_i,m_i}(z), \label{eq:Ualphai}
\end{align}
with $i=1,\dots,n, \ j=m_1,\dots,m_n$, and $P_{i,j}\in\C^{r\times r}$ being an orthogonal rank-one projection.

Now, we can rearrange (\ref{eq:W_1TW}) in the form
\begin{align}\label{eq:Urearrange}
& U_{\alpha_1,m_1-1}^*(z)\cdots U_{\alpha_1,1}^*(z)U^*W_1(z) = U_{\alpha_n}(z)\cdots U_{\alpha_2}(z) U_{\alpha_1,m_1}(z)W(z).
\end{align}
Notice that the left-hand side of (\ref{eq:Urearrange}) is analytic in $\mathscr{A}_p\setminus\mathscr{A}_z$ with (right) inverse analytic in $\mathscr{A}_z\setminus\mathscr{A}_p$. It follows that the right-hand side of (\ref{eq:Urearrange}) must be analytic in $\mathscr{A}_p\setminus\mathscr{A}_z$ with (right) inverse analytic in $\mathscr{A}_z\setminus\mathscr{A}_p$. 
By rewriting the right-hand side of (\ref{eq:Urearrange}) in a more explicit way, we obtain
\begin{align*}
U_{\alpha_n}(z)\cdots U_{\alpha_2}(z) U_{\alpha_1,m_1}(z)W(z)&\stackrel{(\ref{eq:Ualphaij})}{=}U_{\alpha_n}(z)\cdots U_{\alpha_2}(z)\left(I_r-P_{1,m_1} +\frac{1-\overline{\alpha}_1z}{z-\alpha_1}P_{1,m_1}\right)W(z)\\
&=U_{\alpha_n}(z)\cdots U_{\alpha_2}(z)\frac{1-\overline{\alpha}_1z}{z-\alpha_1}P_{1,m_1}W(z) +\Delta(z)
\end{align*}
where $\Delta(z):=U_{\alpha_n}(z)\cdots U_{\alpha_2}(z)(I_r-P_{1,m_1})W(z)$ has no pole at $\alpha_1$. In fact, $\alpha_1\in\mathscr{A}_p\setminus\mathscr{A}_z$ so that $W(z)$ does not have a pole at $\alpha_1$. The minimality of the factorization of $T(z)$ in (\ref{eq:T(z)}) implies that
\[
\left(\frac{1-\overline{\alpha}_1z}{z-\alpha_1}P_{1,m_1}\right)W(z)
\] 
must have a zero-pole cancellation at $\alpha_1$. This fact needs a detailed explanation.

First, define
$U_{\mathrm{res}}:= U_{\alpha_n}(z)\cdots U_{\alpha_2}(z)$ and
 notice that, since the factorization of $T(z)$
in (\ref{eq:T(z)}) is minimal, the matrix 
\begin{align*}
R(z)&:={U}_{\mathrm{res}}(z)\frac{1-\overline{\alpha}_1z}{z-\alpha_1}P_{1,m_1}
\end{align*}
has a pole at $\alpha_1$. In fact, a pole cancellation at $\alpha_1$ in $R(z)$ would imply a pole cancellation at $\alpha_1$ in $U_{\alpha_n}(z)\cdots U_{\alpha_2}(z)U_{\alpha_1,m_1}(z)$, yielding that the degree of the pole $\alpha_1$ in $T(z)$ is less than $m_1$. However, this is not possible since, by Lemma \ref{Lemma1}, the factorization in (\ref{eq:T(z)}) is minimal.
Now, 
since $P_{1,m_1}$ is an orthogonal rank-one projection, there exists a unitary matrix $Q\in\C^{r\times r}$ such that
\[
Q^* P_{1,m_1}Q = \mathrm{diag}[1,0,\dots,0].
\]
Since $Q$ is constant and nonsingular, also 
\begin{align}\label{eq:R'}
\tilde{R}(z)&:=R(z)Q={U}_{\mathrm{res}}(z)\frac{1-\overline{\alpha}_1z}{z-\alpha_1}QQ^*P_{1,m_1}Q\notag\\
&={U}_{\mathrm{res}}(z)Q\mathrm{diag}\left[\frac{1-\overline{\alpha}_1z}{z-\alpha_1},0,\dots,0\right]
\end{align}
has a pole at $\alpha_1$. More in detail,  at least one entry in the first column of $\tilde{R}(z)$ possesses a pole at $\alpha_1$, while all the other columns are identically zero.
Now consider
$
A(z):=R(z)W(z)$ 
\begin{align}\label{eq:diagT}
A(z)&:=R(z)W(z)=\tilde{R}(z)\tilde{W}(z),
\end{align}
where $\tilde{W}(z):=Q^*W(z).$
As already observed, $A(z)$ is analytic in $\mathscr{A}_p\setminus\mathscr{A}_z$.
Therefore, by taking into account that at least one entry in the first column of $\tilde{R}(z)$ possesses a pole at $\alpha_1$, while all the other columns are identically zero, it is immediate that every element in the first row of $\tilde{W}(z)$ has a zero at $\alpha_1$ or is identically zero. Then, 
\[
Q\,\mathrm{diag}\left[\frac{1-\overline{\alpha}_1z}{z-\alpha_1},0,\dots,0\right]\tilde{W}(z) =\left(\frac{1-\overline{\alpha}_1z}{z-\alpha_1}P_{1,m_1}\right) W(z)
\] 
has no pole at $\alpha_1$. This implies that also $U_{\alpha_1,m_1}(z)W(z)$ has no pole at $\alpha_1$ so that in the product $U_{\alpha_1,m_1}(z)W(z)$ there is a pole  cancellation at $\alpha_1$.
Eventually, since $U_{\alpha_1,m_1}(z)$ has full (column-)rank and $W(z)$ has full row-rank, by Lemma \ref{Lemma-zero-pole-cancellation}, we can conclude that in the product $U_{\alpha_1,m_1}(z)W(z)$ there is a zero-pole  cancellation at $\alpha_1$.

By replacing (\ref{eq:Urearrange}) with
\begin{align*}
&W_1^{-R}(z)U U_{\alpha_1}(z)\cdots U_{\alpha_1,m_1-1}(z)=W^{-R}(z)U^*_{\alpha_1,m_1}(z)U_{\alpha_2}^*(z)\cdots U_{\alpha_n}^*(z),
\end{align*}
we can repeat almost verbatim the previous argument in order to conclude that $W^{-R}(z)U^*_{\alpha_1,m_1}(z)$ must have a zero-pole cancellation at $1/\overline{\alpha}_1$, or, equivalently, $U_{\alpha_1,m_1}(z)W(z)$ must have a zero-pole cancellation at $1/\overline{\alpha}_1$.

The zero-pole cancellations at $\alpha_1$ and at $1/\overline{\alpha}_1$ in the product $U_{\alpha_1,m_1}(z)W(z)$ imply that 
\[
\delta_M(U_{\alpha_1,m_1}W)<\delta_M(W)=\frac{1}{2}\delta_M(\Phi).
\] 
Indeed, let $p_1,\dots,p_h\in\overline{\C}$ be the poles of $W(z)$ s.t. $p_i\neq 1/\overline{\alpha}_1$ for all $i=1,\dots,h$. Since $U_{\alpha_1,m_1}(z)$ is analytic together with its inverse in $\overline{\C}\setminus\{\alpha_1,1/\overline{\alpha}_1\}$, it holds 
\[
\delta_p(W;p_i)=\delta_p(U_{\alpha_1,m_1}W;p_i)
\]
 for all $i=1,\dots,h$.
 Moreover, by the zero-pole cancellations: (i) $U_{\alpha_1,m_1}(z)W(z)$ has no pole at $\alpha_1$, and (ii) $\delta_p(U_{\alpha_1,m_1}W;1/\overline{\alpha}_1)<\delta_p(W;1/\overline{\alpha}_1)$. Therefore 
\begin{align*}
\delta_M(U_{\alpha_1,m_1}W)&=\sum_{i=1}^h\delta_p(W;p_i) + \delta_p(U_{\alpha_1,m_1}W;1/\overline{\alpha}_1) \\
&<\sum_{i=1}^h\delta_p(W;p_i) + \delta_p(W;1/\overline{\alpha}_1)=\delta_M(W).
\end{align*}  But this is clearly not possible since, by point 3), $W(z)$ is a stochastically minimal spectral factor.  Therefore,  $U_{\alpha_1,m_1}(z)$ must be a constant unitary matrix.

The previous reasoning still applies for all the other factors of $T(z)$ having a pole at $\alpha_i$, $i=1,\dots,t$, yielding that $m_i=0$ for all $i=1,\dots,t$, i.e., $T(z)$ has no poles at $\alpha_i$, $i=1,\dots,t$.

It remains to show that $T(z)$ has no pole at $\alpha_{t+1},\dots,\alpha_n$. To this aim, 
we have
$$
W_1(z) = T(z) W(z)
$$
and since all the poles of $T(z)$ lie in $\mathscr{A}_z\setminus \mathscr{A}_p$, by Lemma \ref{Lemma-sum-polar-degree}, we have $\delta_p(W_1;\alpha_i)=\delta_p(T;\alpha_i)+\delta_p(W;\alpha_i)$ for all $i=t+1,\dots,n$, while for all the other poles $p_i$, $i=1,\dots,h$, of $W(z)$, $\delta_p(W_1;p_i)=\delta_p(W;p_i)$. 
This implies that
\begin{align*}
\delta_M(W_1)&=\sum_{i=t+1}^n\delta_p(T;\alpha_i)+\sum_{i=t+1}^n\delta_p(W;\alpha_i)+ \sum_{i=1}^h\delta_p(W;p_i)\\
&>\sum_{i=t+1}^n\delta_p(W;\alpha_i)+\sum_{i=1}^h\delta_p(W;p_i)\\
&=\delta_M(W),
\end{align*}
which, by virtue of the stochastic minimality of $W_1(z)$, leads to a contradiction. Hence $T(z)$ must have no poles at $\alpha_{t+1},\dots,\alpha_n$.

To conclude, we have shown that $T(z)$ has no poles and hence no zeroes, due to the fact that $T(z)$ is a para-unitary matrix. Therefore, since it has real entries, $T(z)$ must be a constant orthogonal matrix.
\end{proof}

\section{Conclusions}

In this paper we have analyzed uniqueness of the solution of spectral factorization problem with prescribed dynamical features. 
If we restrict attention to solutions of minimal complexity, the solution is indeed essentially unique. 
The proof is based on the parametrization of discrete-time all-pass functions provided in \cite{unitary2} and on some preliminary results on rational matrix functions that we have established and that may be of independent interest.

We believe that similar techniques may be employed to derive the continuous-time counterpart of this result. Indeed, in \cite{unitary2} a parametrization of continuous-time all-pass functions is also provided and the rest of the procedure appears to be adaptable to the continuous-time case as well.

\end{document}